\documentclass[reqno,12pt]{article}
\usepackage{amsfonts}
\usepackage{latexsym}
\usepackage{mathtools}
\usepackage{bm}
\usepackage{amsmath,amsthm,amssymb}
\usepackage{authoraftertitle}
\usepackage{setspace}
\usepackage{geometry}
\usepackage{authorindex}
\usepackage{graphicx}
\usepackage{overpic}
\usepackage{color,xcolor}
\usepackage{overpic}
\usepackage{slashed}
\usepackage{fancyhdr} 
\usepackage{url,cite}
\usepackage[colorlinks,linkcolor=cyan,anchorcolor=green,citecolor=blue]{hyperref}

\makeatletter
\@addtoreset{equation}{section}
\makeatother

\allowdisplaybreaks[1]
\theoremstyle{plain}
\newtheorem{thm}{Theorem}[section]
\newtheorem{lem}[thm]{Lemma}
\newtheorem{cor}[thm]{Corollary}
\newtheorem{prop}[thm]{Proposition}

\newtheorem*{fact}{Fact}
\theoremstyle{remark}
\newtheorem*{remark}{Remark}

\title{\Large \bf \boldmath\ \\  The boundary correspondence under quasiconformal mappings and VMO-Teichm\"{u}ller  space$^{\ast}$}

\author{\large  Tailiang LIU$^1$\qquad  Yuliang SHEN$^{2}$}

\date{}

\begin{document}

\maketitle

\renewcommand{\thefootnote}{\fnsymbol{footnote}}

\footnotetext{\hspace*{-5mm} \begin{tabular}{@{}r@{}p{13.4cm}@{}}
$^1$ & School of Mathematics and Physics, Jiangsu University of Technology, Changzhou 213001, P. R. China.\\
&{E-mail:}  ltlmath@163.com\\
$^{2}$ & Department of Mathematics, Soochow University,
Suzhou 215006, P. R. China.\\
&{E-mail:}   ylshen@suda.edu.cn\\
 $^{\ast}$ & Project supported by
  the National Natural Science Foundation of China (Nos. 12401095 and 12171346).
\end{tabular}}

\renewcommand{\thefootnote}{\arabic{footnote}}

\begin{abstract}
  In this paper, we introduce  a class of  vanishing Carleson measures  with conformal invariance and corresponding strongly vanishing symmetric homeomorphisms  on the real line and prove that they can be mutually generated under  quasiconformal mappings. This is    motivated by  constructing  a nice  {VMO}-Teichm\"{u}ller  space on the real line, which completely removes the obstacle of the usual VMO-Teichm\"{u}ller  space that lacks conformal invariance and is repeatedly encountered in the papers \cite{Sh,WM1,WM2,WM3}.

\end{abstract}

 \setlength\abovedisplayskip{5pt} \setlength\belowdisplayskip{5pt}

\section{Introduction}

We first fix some basic notations, more detailed definitions can be found in the next section. Let $\mathbb{H}$ and $\mathbb{L}$ denote  the upper and lower half plane in the complex plane $\mathbb{C}$ respectively,  $\mathbb{D}$ and $\mathbb{D}^*$ denote  the interior and exterior of the unit circle $S^1$. Hom$^+(\mathbb{R})$ denotes the group of all sense-preserving homeomorphisms of the real line  $\mathbb{R}$ onto itself. In addition, unless otherwise specified, we always  restrict the various concepts  to  $\mathbb{R}$ or $\mathbb{H}$ in the following text.

Given  a quasiconformal (q.c.)  mapping $\rho$ of the upper half plan $\mathbb{H}$  onto itself,  a basic problem is to understand how  boundary properties of $\rho$ are reflected in its complex dilatation $\mu=\mu_{\rho}$.
 Beurling and Ahlfors  in \cite{BA} proved that  a homeomorphism $h\in\text{Hom}^+(\mathbb{R})$  is  quasisymmetric if and only if there exists some q.c. map  of $\mathbb{H}$ onto itself which has boundary values $h$. Here, a homeomorphism $h\in\text{Hom}^+(\mathbb{R})$  is quasisymmetric and belongs to the class QS$(\mathbb{R})$ if  there exists a  positive constant $C$  such that $C^{-1}\leq\Delta(x,t)\leq C $ for all $x\in\mathbb{R}$ and $t>0$, where
\[\Delta(x,t)\doteq\frac{h(x+t)-h(x)}{h(x)-h(x-t)}.\]
The set of all normalized quasisymmetric homeomorphisms (say, fix the points $0$ and $1$), denoted by $T=\text{QS}(\mathbb{R})/\sim$ , is one of the models of the universal Teichm\"{u}ller space,  a universal parameter space of for all Riemann surfaces.
A  homeomorphism $h$ in QS$(\mathbb{R})$ is  symmetric and belongs to the class S$(\mathbb{R})$ if $\lim_{t\rightarrow0}\Delta(x,t)=1$ uniformly for all $x\in\mathbb{R}$.
This class was first studied  by Carleson in \cite{Ca} to discuss the absolute continuity of a quasisymmetric homeomorphism. It was investigated in depth later by Gardiner and Sullivan in \cite{GS} to study little and  asymptotic Teichm\"{u}ller spaces.  In particular, it was proved that $h\in \text{S}(\mathbb{R})$ if and only if $h$ can be extended as an asymptotically conformal mapping $f$ to $\mathbb{H}$, and that the Beurling-Ahlfors extension of $h$ is asymptotically conformal when $h$ is symmetric (see\cite{Ca,GS}).
The meaning of asymptotically conformal mapping here is that its complex  dilatation  tends to zero at the boundary.  The symmetric Teichm\"{u}ller space is defined as $T_{*}(\mathbb{R})=\text{S}(\mathbb{R})/\sim$.

In this paper, we will introduce and discuss a new subclass of QS$(\mathbb{R})$, which we call {strongly vanishing symmetric homeomorphisms}. For this purpose, let us recall some related  notions.
A   homeomorphism $h\in\text{Hom}^+(\mathbb{R})$  is called strongly qusisymmetric  in the sense of Semmes \cite{S2} and belongs to the class SQS$(\mathbb{R})$, if there exist $0<C_{1}, C_{2}<\infty$ such that for all intervals $I$ in $\mathbb{R}$  and all measurable subsets $A$ of $Q$ we have
\[{|h(A)|}/{|h(I)|}<C_{1}(|A|/|I|)^{C_{2}}.\] Equivalently, $h\in\text{SQS}(\mathbb{R})$   if and only if $h$ is locally absolutely continuous so that $h'$ belongs to the class $A_{\infty}$ weights  introduced by Muckenhoupt (see \cite{CF}), in particular, $\log h'\in$ BMO$(\mathbb{R})$,  the space  of functions on
$\mathbb{R}$ of bounded mean oscillation .    The class  is much investigated  because of its great importance in the applications to the  BMO theory of the universal Teichm\"{u}ller space and elliptic operator theory.  It was proved that   $h\in \text{SQS}(\mathbb{R})$   if and only if it can be extended as a q.c. map on $\mathbb{H}$  with complex dilatation $\mu$ such that $|\mu(z)|^2/ydxdy$ is a Carleson measure (see \cite{FKP,S2}). It is known that  SQS$(\mathbb{R})$ is a subgroup of $\text{QS}(\mathbb{R})$, and $T_b(\mathbb{R})=\text{SQS}(\mathbb{R})/\sim$  is called (the real line model of ) the BMO-Teichm\"{u}ller space.
Now, we say that a homeomorphism $h\in\text{Hom}^{+}(\mathbb{R})$ is strongly symmetric and belongs to the class
SS$(\mathbb{R})$ if  $h\in\text{SQS}(\mathbb{R})$ and $\log h'\in\text{VMO}(\mathbb{R})$, the space of functions on
$\mathbb{R}$ of vanishing mean oscillation (see \cite{Gar}).  This class was introduced by Shen  in \cite{Sh}  where the original VMO-Teichm\"{u}ller space $T_{v}(\mathbb{R})=$SS$(\mathbb{R})/\sim$ is  investigated on the real line. It was proved that $h\in \text{SS}(\mathbb{R})$ if and only if $h$ can be extended as a q.c. mapping to  $\mathbb{H}$ whose complex dilatation $\mu$ induces a   vanishing Carleson measure $|\mu(z)|^2/ydxdy $ (see \cite{Sh,WM2}). We may define {SQS}$(S^1)$ and {SS}$(S^1)$ in the same way and have a parallel BMO theory of the universal Teichm\"uller space on the  unit circle  which was  developed in the papers \cite{AZ} and \cite{SW}.

Denoting by  CMO$(\mathbb{R})$ the completion of continuous functions on the real line with compact support in the $ \text{BMO} $ topology, while  the vanishing mean oscillation function space $\text{VMO}(\mathbb{R})$ we usually refer to is not subjected to the compactness.  The  space CMO$(\mathbb{R})$ has been studied for a rich harmonic analysis history because of its irreplaceability (see  for example \cite{BLS,Bo,CW,Ne}).
Now, we say that a homeomorphism $h\in\text{Hom}^{+}(\mathbb{R})$  is   strongly vanishing symmetric and belongs to the class $\text{SS}_{0}(\mathbb{R})$  if $h\in\text{SQS}(\mathbb{R})$  and $\log h'\in$ CMO$(\mathbb{R})$. Then  $\text{SS}_{0}(\mathbb{R})\subseteq \text{SS}(\mathbb{R})$.

 We shall be concerned with the boundary correspondence problem of characterizing $\text{SS}_{0}(\mathbb{R})$ under  q.c. extensions. Our main result can be stated as

 \begin{thm} \label{1.1}
	A homeomorphism $h\in\text{Hom}^{+}(\mathbb{R})$ is strongly vanishing symmetric if and only if it can be extended as a q.c. mapping to  $\mathbb{H}$ whose complex $\mu$ induces    a strongly vanishing Carleson measure $|\mu(z)|^2/ydxdy $ on $\mathbb{H}$.
\end{thm}

 Another goal of the paper is (to use  Theorem \ref{1.1})   to remove the  obstacle of $T_{v}(\mathbb{R})$ that lacks conformal invariance and  is repeatedly encountered in the articles \cite{Sh,WM1,WM2,WM3}.  For that,  we shall describe some details of what is conformal invariance of  Teichm\"{u}ller spaces,  which also  play a fundamental role in  the rest of this paper.

 Let $\text{CM}(\Omega)$ and $\text{CM}_{0}(\Omega)$  denote the set of all Carleson measures and vanishing Carleson measures defined on the simple connected domain $\Omega$, respectively.
  It is known that  a positive  measure  $ \lambda\in\text{CM}(\mathbb{D}) $ if and only if $(\gamma^*\lambda)/|\gamma'| \in\text{CM}(\mathbb{H})$ (see \cite{Gar}), and we call this conformal invariance of Carleson measures. Here and in what follows  $\gamma(z)$ always denotes the Cayley transform $(z-i)/(z+i)$ from the upper half plane $\mathbb{H}$ onto the unit disk $\mathbb{D}$. This directly leads to the equivalence of $|\mu_{H}(w)|^2/(1-|w|^2)dudv\in\text{CM}(\mathbb{D})$ and $|\mu_{\gamma^{-1}\circ H\circ\gamma}(z)|^2/ydxdy\in\text{CM}(\mathbb{H})$, where $H$ is a q.c. map of $\mathbb{D}$ onto itself.
   Consequently,  we obtain that for a homeomorphism $h\in\text{Hom}^{+}(\mathbb{R})$, $h\in\text{SQS}(\mathbb{R})$ if and only if $\gamma\circ h\circ\gamma^{-1}\in\text{SQS}(S^1)$ by the corresponding q.c. extension. Then, $T_{b}(\mathbb{R})$ and $T_{b}(S^1)$ (the unit circle model of the BMO-Teichm\"uller space) can be converted to each other  under the conjugate action of  the Cayley transform,  which is  what we call the conformal invariance of the  BMO-Teichm\"{u}ller space.

 However, VMO and  vanishing Carleson measures lack conformal invariance (see \cite[Section 6]{WM3}), which destroys the desirable properties of  the VMO-Teichm\"{u}ller space. Jones proved that a homeomorphism $h$ is strongly quasisymmetric  if and only if  the pull-back operator $P_{h}(f)=f\circ h$ is a bounded   isomorphism of the BMO space (see\cite{Jon}). Noting that  VMO($S^1$) is the closure of continuous functions on the unit circle  in the BMO norm, it is easy to see that $P_{h}$ preserves VMO($S^1$). But  VMO($\mathbb{R}$)  is not  preserved  by $P_{h}$ (see \cite{WM3}), which indicates that $\text{SS}(\mathbb{R})$ has no group structure.
 More importantly,  there exists $h\in\text{SS}(\mathbb{R})$ such that  $\gamma\circ h\circ\gamma^{-1}\notin\text{SS}(S^1)$ (see \cite[Theorem 19]{WM3}), thus $T_{v}(\mathbb{R})$ and   $T_{v}(S^1)$ (the unit circle model of the VMO-Teichm\"uller space) are not well compatible.

We believe that $T_{v}(\mathbb{R})$ is only a product of harmonic analysis, which has left us in the shadow of the fundamental properties of the Teichm\"{u}ller space for a long time. We hope to find a new VMO-Teichm\"{u}ller space with conformal invariance and group structure to replace $T_{v}(\mathbb{R})$.
The introduction of SS$_{0}(\mathbb{R})$ will reconcile these problems.  Let's first state that $\text{SS}_0(\mathbb{R})$ forms a group.
Notice that  the pull-back operator $ P_h $ preserves $ \text{CMO}(\mathbb{R}) $ by the usual density argurments if $h\in\text{SQS}(\mathbb{R})$. Then for any $ f,g\in \text{SS}_0(\mathbb{R}) $ we obtain that $ \log(g\circ f^{-1})'=\log g'\circ f^{-1}-\log f'\circ f^{-1}\in\text{CMO}(\mathbb{R})$, which implies that $ g\circ f^{-1}\in \text{SS}_0(\mathbb{R}) $. Now, set $ T_{c}(\mathbb{R})=\text{SS}_0(\mathbb{R})/\sim $, and call it the CMO-Teichm\"{u}ller space on the real line. Then the natural conformally invariant conditions on the Bers model (or the complex dilatation model) for  $ T_{c}(\mathbb{R})$ turn out to be in terms of the conformal invariance of  strongly vanishing Carleson measures $\text{CM}_{s}(\mathbb{H}) $ (see section 3). Consequently, we can get that  $T_{c}(\mathbb{R})$ and $T_v(S^1)$  are compatible by Theorem \ref{1.1} and Theorem \ref{3.7} below. More precisely, we have

\begin{thm} \label{1.2}
	A homeomorphism $h\in\text{Hom}^{+}(\mathbb{R})$ is strongly vanishing symmetric if and only if $\gamma\circ h\circ\gamma^{-1}$ is strongly symmetric on the unit circle.
\end{thm}

The paper is broken into the following sections. In section 2, we  revisit some basic definitions and facts.  In section 3, we will prove the   conformal invariance of $\text{CMO}(\mathbb{R}) $  and $\text{CM}_{s}(\mathbb{H}) $ to prepare for the boundary correspondence problem.  Section 4 is devoted to the proof of Theorem \ref{1.1}.

In the paper, $C$, $C_1$, $C_2$ $\cdots$ will denote universal constants that might change from one line to another. The notation $A\lesssim B$ $(A\gtrsim B)$ means that there is an independent constant $C$ such that $A\le CB$ $ (A\ge CB)$. The notation $A \simeq B$ means both $A\lesssim B$ and $A\gtrsim B$. Also, as usual, $\chi_{E}$   denotes the characteristic function on a set $E$. For $a>0$, we define the dilate $aI$ of an interval $I$ to be the interval with the same midpoint as $I$ and length $a|I|$.  Given $x_{0}\in\mathbb{R}$, we define the translation $I+x_{0}$ $\doteq\{x+x_{0}:x\in I\} $ of an interval $I$.

\section{Basic facts and definitions}

A locally integrable function $f$ on $\mathbb{R}$ lies in the space BMO$(\mathbb{R})$ if
\[\|f\|_{*}=\sup_{I\subset\mathbb{R}}\frac{1}{|I|}\int_{I}|f-f_{I}|dx<\infty,\] where  $f_{I}$ denotes the integral mean of $f$ over $I$. It is well-known that John-Nirenberg inequality  implies that for $1\leq r<\infty$
\[\|f\|^{r}_{*}\simeq \sup_{I\subset\mathbb{R}}M_{r,f}(I),\quad\text{where }M_{r,f}(I)\doteq \frac{1}{|I|}\int_{I}|f-f_I|^rdx.\] We say a BMO function belongs to the space VMO$(\mathbb{R})$ if
\[\lim_{|I|\rightarrow 0}\frac{1}{|I|}\int_{I}|f-f_{I}|dx=0.\] It is well-known that VMO$(\mathbb{R})$ is the  closure of uniformly continuous functions on $\mathbb{R}$ in the $ \text{BMO} $ topology. Similarly, we can define BMO $(S^1)$ and VMO$(S^1)$ on the unit circle $S^1$.

 A holomorphic function $F$ on $\mathbb{H}$  belongs to the space BMOA($\mathbb{H}$) if there exists some function $f\in\text{BMO}(\mathbb{R})$ whose Poission extension $P_{f}$ is precisely $F$.  Recall the Poisson extension of $f$ to $\mathbb{H}$ is defined as
\[P_f(z)\doteq\int_{\mathbb{R}} P(z,t)f(t)dt=\frac{1}{\pi}\int_{\mathbb{R}}\frac{y}{(x-t)^2+y^2}f(t)dt, \  z=x+iy\in\mathbb{H}.\] Similarly, we say that a holomorphic function $G$ belongs to BMOA($\mathbb{D}$) if there exists some function $g\in\text{BMO}(S^1)$ such that the Poission extension $P_{g}$ of $g$ to $\mathbb{D}$ is precisely $G$, where \[P_g(w)\doteq\int_{S^1}P(w,\theta)g(e^{i\theta})d\theta=\frac{1}{2\pi}\int_{0}^{2\pi}\frac{1-|w|^2}{|e^{i\theta}-w|^2}g(e^{i \theta})d\theta,  \  w\in\mathbb{D}.\] Notice that   the Poisson kernel enjoys conformal  invariance (see\cite{Gar}), i.e. for the Cayley transform $ w=\gamma(z)$ and $g\in\text{BMO}(S^1)$  we have $ P_g(w)=P_{g\circ\gamma}(z) $, and this fact will be frequently used   in the paper. In addition, if one imposes the  corresponding vanishing condition, one naturally obtains the definition of the VMOA$(\mathbb{H})$, CMOA$(\mathbb{H})$ and VMOA$(\mathbb{D})$. We may define these spaces on $\mathbb{L}$ and $\mathbb{D^*}$ in the same way.

For each interval $I$ on $\mathbb{R}$, the associated Carleson square $Q_I$ is defined as $ $
$Q_I=\{(x,y)\in\mathbb{H}: x\in I, y<|I| \}.$
A positive measure $\lambda$ on $\mathbb{H}$ is called a Carleson measure and belongs to the class $\operatorname{CM}(\mathbb{H})$  if
\[\|\lambda\|^2_{c}\doteq\sup_{I\subset\mathbb{R} }\frac{\lambda(Q_I)}{|I|}<\infty.\]
A Carleson measure $ \lambda$ is  said to be vanishing  and belongs to the class $\operatorname{CM}_{0}(\mathbb{H})$ if \[\lim_{|I|\rightarrow0}\frac{\lambda(Q_I)}{|I|}=0.\]
 Now, we say a Carleson measure  is strongly vanishing and belongs to the class $\operatorname{CM}_{s}(\mathbb{H})$ if  \[\lim_{|I|\rightarrow0}\frac{\lambda(Q_I)}{|I|}=\lim_{|I|\rightarrow\infty}\frac{\lambda(Q_I)}{|I|}=
 \lim_{x\rightarrow\infty}\frac{\lambda(Q_{I+x})}{|I|}=0. \]
We can define these classes on the lower half plane $\mathbb{L}$ in the same way. Carleson measures are closely connected to BMO. By Fefferman-Stein, a function $f$ on $\mathbb{R}$ lies in BMO$(\mathbb{R})$ (VMO$(\mathbb{R})$) if and only if $|\nabla P_{f}(z)|^{2}y dm(z)$ is a (vanishing) Carleson measure, where $dm(z)$ denotes two-dimensional Lebesgue measure $dxdy$. The case of $S^1$ is similar by the same proof. For details, see Garnett's book \cite{Gar} and Zhu's \cite{Zh}.

 A function $w\geq0$ belongs to the class $A_{p}$ $(1< p<\infty)$ if
  \[\sup_{I\subset\mathbb{R}}\Big(\frac{1}{|I|}\int_{I}w(x)dx\Big)\Big(\frac{1}{|I|}\int_{I}w(x)^{-\frac{1}{p-1}}dx\Big)^{p-1}<\infty.\]
It is well-known that  $e^f\in A_{p}\subseteq A_{\infty}$ if $f$ is real-valued and $\|f\|_{*}$ is small enough, which is a simple estimation by John-Nirenberg inequality. As far as CMO is concerned, everything remains the same.

\begin{fact}  \label{fact}
	If $f\in\operatorname{CMO}(\mathbb{R})$ is real-valued, then $e^f\in A_{p}\subseteq A_{\infty}$ .
\end{fact}
\begin{proof}
	For convenience, we restrict the fact to $p=2$, and the general case has been compressed in the following proof.
	
	Since $f\in\operatorname{CMO}(\mathbb{R}) $, 	there  exist $g\in C_{0}(\mathbb{R})$ and $h\in\operatorname{BMO}(\mathbb{R})$ with  $\|h\|_{*}<\varepsilon$ such that $f=g+h$.
	
	By John-Nirenberg,
		\begin{align*}
		\frac{1}{|I|}\int_{I}e^{|h(x)-h_I|}dx&=\frac{1}{|I|}\int^\infty_{0}\big|\{z\in I: |h(x)-h_I|>t\}\big|e^{t}dt+1\\
		&\leq C_1 \int^\infty_{0}\exp\Big(\frac{-C_2 t}{\|h\|_{*}}\Big)e^{t}dt+1\\
		&=\frac{C_1\|h\|_{*}}{C_2-\|h\|_{*}}+1.
	\end{align*}
	Now, we consider \begin{align*}
		A(I)&\doteq\int_{I}e^{g+h}dx\int_{I}e^{-(g+h)}dx\\
		&=\int_{I}e^{g-g_I+h-h_I}dx\int_{I}e^{g_I-g+h_I-h}dx\\
		&\leq \Big(\int_{I}e^{2(g-g_I)}dx\int_{I}e^{2(h-h_I)}dx\Big)^{\frac{1}{2}}\Big(\int_{I}e^{2(g_I-g)}dx\int_{I}e^{2(h_I-h)}dx\Big)^{\frac{1}{2}}		\\
		&\leq\int_{I}e^{2|g(x)-g_I|}dx\int_{I}e^{2|h(x)-h_I|}dx\\
		&\leq|I|^2e^{\|2g\|_{\infty}} \Big(\frac{2C_1\varepsilon}{C_2-2\varepsilon}+1\Big).
	\end{align*} Hence, $ e^f\in A_{2}\subseteq A_{\infty}$.
\end{proof}
\begin{remark}
	This fact implies that the definition of $\operatorname{SS}_{0}(\mathbb{R})$ can be improved to the set of all locally absolutely continuous homeomorphisms $h$ with $\log h'\in\text{CMO}(\mathbb{R})$, which is  different from the case of $\operatorname{SS}(\mathbb{R})$.
\end{remark}

\section{Conformal invariance of $\text{CMO}(\mathbb{R}) $  and $\text{CM}_{s}(\mathbb{H}) $}

The conformal invariance of BMO$(\mathbb{R})$, which has been exploited by many authors (see for example Garnett's \cite{Gar}), seems first to appear in Garsia's  unpublished notes in 1971. However, the conformal invariance of CMO$(\mathbb{R})$ and the corresponding strongly vanishing Carleson measures $\operatorname{CM}_{s}(\mathbb{H})$ are rarely mentioned.  Due to  the importance of  CMO space and its fundamental role in the VMO-Teich\"{u}ller space theory,   we fill this gap in this section.

As we mentioned earlier, CMO$(\mathbb{R})$  is defined as the closure of continuous functions with compact support in the $ \text{BMO} $ topology.  In fact,  $\text{CMO}(\mathbb{R}) $ can be  characterized  by three vanishing conditions:
\begin{lem}
	\label{3.1} Let $f\in\operatorname{BMO}(\mathbb{R})$.  Then  the following statements are all equivalent: \begin{align*}
		(1) &f\in\operatorname{CMO}(\mathbb{R}),\\
		(2) &\lim_{|I|\rightarrow0} M_{1,f}(I)= \lim_{|I|\rightarrow\infty} M_{1,f}(I)= \lim_{x\rightarrow\infty} M_{1,f}(I+x)=0,\\
		(3) & \lim_{|I|\rightarrow0} M_{r,f}(I)= \lim_{|I|\rightarrow\infty} M_{r,f}(I)= \lim_{x\rightarrow\infty} M_{r,f}(I+x)=0, \quad(1<r<\infty).
	\end{align*}
\end{lem}
\begin{remark}
	 The implication $(1)\Longrightarrow(2)$  or $(1)\Longrightarrow(3)$ is clear by the definition.  $(2)\Longrightarrow(1)$  was announced in Neri \cite{Ne}, and Uchiyama in\cite{Uc} formally gave a rigorous proof of the result.  Finally, we  easily see that $(3)\Longrightarrow(2)$ by H\"{o}lder's inequality.
\end{remark}

 We  now give another characterization of  $\text{CMO}(\mathbb{R}) $, which is also the rudiment of conformal invariance of the space.
The proof is inspired by Garnett in the case of BMO$(\mathbb{R})$ (see\cite{Gar}), but it's  more complicated than that.

In what follows, $M_{r,P_f}(z)$ is defined by
\[M_{r,P_f}(z)\doteq\int_{\mathbb{R}}|f(t)-P_f(z)|^r P(z,t)dt.\]
\begin{prop}\label{3.2}
	 Let $f\in\operatorname{BMO}(\mathbb{R})$.  Then $f\in\operatorname{CMO}(\mathbb{R})$ if and only if for all $z=x+iy\in\mathbb{H}$,
	 \[\lim_{y\rightarrow0}M_{r,P_f}(z)=\lim_{y\rightarrow\infty}M_{r,P_f}(z)=\lim_{x\rightarrow\infty}M_{r,P_f}(z)=0, \ (r\geq1).\]	
\end{prop}

\begin{proof}
	$\Longleftarrow$  Let $I$ be an interval on $\mathbb{R}$ and $z=x+iy$, where $x$ is the midpoint of  $I$ and $y=\frac{|I|}{2}$. Then some basic inequalities  and $\frac{\chi_I(t)}{|I|}\leq\pi P(z,t)$  show that
\begin{align*}
		M_{r,f}(I)&=\frac{1}{|I|}\int_{I}|f(t)-P_f(z)+P_f(z)-f_{I}|^rdt\\
		&\lesssim  \frac{1}{|I|}\int_{I}|f(t)-P_f(z)|^rdt+\Big(\frac{1}{|I|}\int_{I}|f(t)-P_f(z)|dt\Big)^{r}\\
		&\lesssim \frac{1}{|I|}\int_{I}|f(t)-P_f(z)|^rdt\\
		&=\int_{\mathbb{R}}\frac{\chi_I(t)}{|I|}|f(t)-P_f(z)|^rdt\\
		&\leq M_{r,P_f}(z)= M_{r,P_f}(x+i|I|/2).
	\end{align*} If one takes limits on both sides and then uses Lemma \ref{3.1}, one easily obtains the sufficiency.
	
$\Longrightarrow$	Suppose  $f\in\text{CMO}(\mathbb{R}) $ and $z=x+iy\in\mathbb{H}$. Set $I_k=\{t:|t-x|<2^ky\}$, $k\in\mathbb{N}$.

We by some basic inequalities have
\begin{align*}
M_{r,P_f}(z)&\lesssim \int_{\mathbb{R}}|f(t)-f_{I_0}|^rP(z,t)dt+\int_{\mathbb{R}}|f_{I_0}-P_f(z)|^rP(z,t)dt\\
&=\int_{\mathbb{R}}|f(t)-f_{I_0}|^rP(z,t)dt+ |f_{I_0}-P_{f}(z)|^{r}\\
&\leq\int_{\mathbb{R}}|f(t)-f_{I_0}|^rP(z,t)dt+ \Big(\int_{\mathbb{R}}|f(t)-f_{I_0}|P(z,t)dt\Big)^r\\
&\lesssim\int_{\mathbb{R}}|f(t)-f_{I_0}|^rP(z,t)dt\\
&\doteq B.
\end{align*}

Notice that   $|I_k|=2^{k+1}y $, $P(z,t)\leq Cy^{-1}$  while $t\in I_{0} $, and  $P(z,t)\leq C(2^{2k}y)^{-1}$ while  $t\in I_k/I_{k-1}$, then we have
\begin{align*}
	B&\lesssim \frac{1}{y}\int_{I_0}|f(t)-f_{I_0}|^rdt+\sum_{k=1}^{\infty}\frac{1}{2^{2k}y}\int_{I_k/I_{k-1}}\big(|f(t)-f_{I_k}|^r+|f_{I_k}-f_{I_0}|^r\big)dt\\
	&\leq \frac{1}{y}\int_{I_0}|f(t)-f_{I_0}|^rdt+ \sum_{k=1}^{\infty}\frac{1}{2^{2k}y}\int_{I_k}|f(t)-f_{I_k}|^rdt+\sum_{k=1}^{\infty}\frac{1}{2^k}|f_{I_k}-f_{I_0}|^r\\
	&\doteq B_1+ B_2 +B_3.
\end{align*}
To estimate $ \lim_{y\rightarrow0}M_{r,P_f}(z) $ firstly. For each $\varepsilon>0$ there exists an $N>0$ such that \[\sum_{n=N+1}^{\infty}2^{-n}<\sum_{n=N+1}^{\infty}n(n+1)2^{-(n+1)}<\varepsilon.\] Also,
there exists  $\delta>0$ such that  for $ |I_k|<\delta  $ $( k\leq N)$   we have \[\frac{1}{|I_k|}\int_{I_k}|f(t)-f_{I_k}|^rdx<\varepsilon.\]
Clearly, $ \lim_{y\rightarrow0}B_1=0 $. For $ B_2 $,
 \begin{align*}
	B_2&=\sum_{k=1}^{\infty}\frac{1}{2^{2k}y}\int_{I_k}|f(t)-f_{I_k}|^rdt\\
	&\simeq\sum_{k=1}^{\infty}\frac{1}{2^{k}}\frac{1}{|I_k|}\int_{I_k}|f(t)-f_{I_k}|^rdt\\
	&<\sum_{k=1}^{N}\frac{\varepsilon}{2^k}+\sum_{k=N+1}^{\infty}\frac{\|f\|^r_*}{2^k}\\
	&\lesssim\varepsilon.
\end{align*}
 Also,
\begin{align*}
	B_3 &=\sum_{n=1}^{\infty}\frac{1}{2^n}|f_{I_n}-f_{I_0}|^r\\
	&\lesssim\sum_{n=1}^{\infty}\frac{1}{2^n}\sum_{k=1}^{n}|f_{I_k}-f_{I_{k-1}}|^r\\
	&\lesssim\sum_{n=1}^{\infty}\frac{1}{2^n}\sum_{k=1}^{n}\frac{1}{|I_k|}\int_{I_k}|f(x)-f_{I_k}|^rdx\\
	&\leq \sum_{n=1}^{N}\frac{n(n+1)}{2^{n+1}}\varepsilon+\sum_{n=N+1}^{\infty}\frac{n(n+1)}{2^{n+1}}\|f\|^r_*\\
	&\lesssim\varepsilon.
\end{align*}
Hence, $ \lim_{y\rightarrow0}M_{r,P_f}(z)=0 $.

 Repeat the above process,  it is easy to get \[\lim_{y\rightarrow\infty}M_{r,P_f}(z)\lesssim\lim_{y\rightarrow\infty} \sum_{k=1}^{3}B_k= 0.\]

Finally,  the case of $x\to\infty$ is also similar. Since $f\in\text{CMO}(\mathbb{R}) $, for any $\varepsilon>0$ there exists $N>0$ such that for any interval $I$ satisfying $\text{dist}(I,0)>N$, we have $M_{r,f}(I)<\varepsilon$. There exists also an $N$ such that  $\text{dist}(I_k,0)>M$ for any $k\leq N$, thus $M_{r,f}(I_k)<\varepsilon$ ($k\leq N$). Similar to  handle  $B_1$, $B_2$, $B_3$ when $y\to0$, we can complete the proof.
\end{proof}

It is much simpler to characterize VMO$(S^1)$ by Poisson extension, which has been proved by D. Girela (see\cite[Theorem 4.9]{Gi}).

\begin{lem} \label{3.3}
	Let $ f\in\operatorname{BMO}(S^1) $. Then $ f\in\operatorname{VMO}(S^1)$ if and only if
	\[\lim_{|w|\rightarrow1}\int_{S^1}|f(e^{i\theta})-P_f(w)|^rP(w,\theta)d\theta=0,\quad(r\geq1). \]
\end{lem}

Now let's state the conformal invariance of $\operatorname{CMO}(\mathbb{R})$.
\begin{thm} \label{3.4}
	Let $f\in\operatorname{BMO}(S^1)$ and  $\gamma(z)=\frac{z-i}{z+i}$ as before. Then
	$f\in\operatorname{VMO}(S^1)$ if and only if $f\circ\gamma\in\operatorname{CMO}(\mathbb{R})$. In other words,  $\operatorname{CMO}(\mathbb{R})$ and  $\operatorname{VMO}(S^1)$ can be  transformed into each other.
\end{thm}
\begin{proof}
	Set $w=\gamma(z)$, we have $ P_{|f-P_{f}(w)|}(w)=P_{|f\circ\gamma-P_{f}\circ\gamma(z)|}(z)$, i.e.
	\[\int_{\mathbb{R}}|f\circ\gamma(t)-P_{f\circ\gamma}(z)|P(z,t)dt=\int_{S^1}|f(e^{i\theta})-P_f(w)|P(w,\theta)d\theta.\]
	Since $\{z:|\gamma(z)|\rightarrow1\}=\{z:\text{Im}z\rightarrow0\}\cup\{z:\text{Im}z\rightarrow\infty\}\cup\{z:\text{Re}z\rightarrow\infty\}  $, we have that $ f\circ\gamma\in \text{CMO}(\mathbb{R})$ if and only if $ f\in\text{VMO}(S^1) $ by Proposition \ref{3.2} and Lemma \ref{3.3}.
\end{proof}


 Next we consider the conformal invariance of $\operatorname{CM}_{s}(\mathbb{H})$. We first recall the following result (see  \cite{Gi,Zh}).
\begin{lem}\label{3.5}
	For a Carleson measure $\lambda$ on $\mathbb{D}$,
	$\lambda\in \operatorname{CM}_{0}(\mathbb{D})$ if and only if
	\begin{equation*}
		\lim_{|w_{0}|\rightarrow1}\int_{\mathbb{D}}|\tau_{w_{0}}'(z)|d\lambda(w)=0, \   \operatorname{where} \ \tau_{w_{0}}(w)=\frac{w-w_0}{1-\overline{w_0}w},\ w_{0}\in\mathbb{D}.
	\end{equation*}
\end{lem}

In the recent work \cite{Sh}, we showed that  a Carleson measure ${\lambda}\in \operatorname{CM}_{0}(\mathbb{H})$ if and only if 	
	\begin{equation*}
		\lim_{y_{0}\rightarrow0}\int_{\mathbb{H}}|\gamma_{z_{0}}'(z)|d{\lambda}(z)=0, \   \operatorname{where} \ \gamma_{z_{0}}(z)=\frac{z-z_0}{z-\overline{z_0}},\  z_{0}=x_{0}+iy_{0}\in\mathbb{H}.
	\end{equation*}
By the same reasoning, we can prove the following result with the help of the discussion during the proof of Propostion \ref{3.2}.
\begin{prop}\label{3.6}
		A Carleson measure
	$\lambda\in \operatorname{CM}_{s}(\mathbb{H})$ if and only if for all $z_0=x_0+iy_0\in\mathbb{H}$,
	\[\lim_{y_0\rightarrow0}\int_{\mathbb{H}} |\gamma'_{z_0}(z)| d\lambda(z)=\lim_{y_0\rightarrow\infty}\int_{\mathbb{H}} |\gamma'_{z_0}(z)| d\lambda(z)=\lim_{x_0\rightarrow\infty}\int_{\mathbb{H}} |\gamma'_{z_0}(z)| d\lambda(z)=0.\]
\end{prop}

\begin{remark}
	The necessity  is equivalent to  that   for any compact set $K\subset\mathbb{H}$
	\[\inf_{K\subset\mathbb{H}}\sup_{z_0\in\mathbb{H}/K}\int_{\mathbb{H}} |\gamma'_{z_0}(z)| d\lambda(z)=0.\]
\end{remark}

Then, we can obtain  the conformal invariance of  $\operatorname{CM}_{s}(\mathbb{H})$.
\begin{thm}
\label{3.7}	Let $\lambda$ be a Carleson measure on $\mathbb{D}$ and $\gamma(z)=\gamma_{i}(z)=\frac{z-i}{z+i} $. Then \[ \lambda\in\operatorname{CM}_0(\mathbb{D})\Longleftrightarrow \nu\doteq(\gamma^*\lambda)/|\gamma'|\in\operatorname{CM}_s(\mathbb{H}).\]
\end{thm}
\begin{proof}
	$ \Longrightarrow$ Notice that  $\gamma_{z_0}\circ\gamma^{-1}\circ\gamma(z_0)=0 $, then
	\[ \big|\gamma_{z_0}\circ\gamma^{-1}(w)\big|=\Bigg|\frac{w-\gamma(z_0)}{1-\overline{\gamma(z_0)}w} \Bigg|=\big|\tau_{\gamma_{z_0}}(w)\big|.\] For any $ z_0\in\mathbb{H} $, we have
	\begin{align*}
		\int_{\mathbb{H}} \big|\gamma'_{z_0}(z)\big| d\nu(z)&=\int_{\mathbb{D}} \big|\gamma'_{z_0}\circ\gamma^{-1}(w)\gamma'\circ\gamma^{-1}(w)\big||{(\gamma^{-1})'(w)}|^2 d\lambda(w)\\
		&=\int_{\mathbb{D}} \big|(\gamma_{z_0}\circ\gamma^{-1})'(w)\big|d\lambda(w)\\
		&=\int_{\mathbb{D}} \big|\tau_{\gamma_{z_0}}'(w)\big|d\lambda(w).
	\end{align*}
	Noting that $\{z_0:|\gamma(z_0)|\rightarrow1\}=\{z_0:\text{Im}\ z_{0}\rightarrow0\, \text {or} \, \infty\}\cup\{z_0:\text{Re}\ z_{0}\rightarrow\infty\}  $, then we have
	$ \lambda\circ\gamma(z)|\gamma'(z)|\in\text{CM}_s(\mathbb{H}) $ by Lemma \ref{3.5} and Proposition \ref{3.6}.
	
	$ \Longleftarrow$  Similarly we have \[\big|\tau_{w_0}\circ\gamma(z)\big|=\Bigg|\frac{z-\gamma^{-1}(w_0)}{z-\overline{\gamma^{-1}(w_0)}}\Bigg|=\big|\gamma_{\gamma^{-1}(w_0)}(z)\big|.\]
	For any $ w_0\in\mathbb{D}$, we have
	\begin{align*}
		\int_{\mathbb{D}}|\tau_{w_0}'(w)|d\lambda(w)&=\int_{\mathbb{H}} \big|(\tau_{w_0}\circ\gamma)'(z)\big| d\nu(z)\\
		&=\int_{\mathbb{H}}\big|\gamma_{\gamma^{-1}(w_0)}'(z)\big|d\nu(z).
	\end{align*}
	Clearly, $ |w_0|\rightarrow1 $ is equivalent to that $ \text{Im}\gamma^{-1}(w_0)\rightarrow 0,\infty $ or $ \text{Re}\gamma^{-1}(w_0)\rightarrow\infty $, we obtain $ \lambda(w)\in\operatorname{CM}_0(\mathbb{D})$ by Lemma \ref{3.5} and Proposition \ref{3.6} again.
\end{proof}

\begin{thm}
Suppose that $f$ belongs to $ \operatorname{BMO}(\mathbb{R}) $. Then
\[f\in\operatorname{CMO}(\mathbb{R}) \Longleftrightarrow  |\nabla P_f(z)|^2ydm(z)\in\operatorname{CM}_{s}(\mathbb{H}).\]
\end{thm}
\begin{proof}
	Set $\alpha=\gamma^{-1}$.
	By Fefferman-Stein and  Theorem \ref{3.4},
	\[ f\in\text{CMO}(\mathbb{R})\Longleftrightarrow f\circ\alpha\in\text{VMO}(S^1)\Longleftrightarrow|\nabla P_{f\circ\alpha}(w)|^2 (1-|w|^2)dm(w)\in \text{CM}_{0}(\mathbb{D}),\]
		while by Theorem \ref{3.7}, \[|\nabla P_f(z)|^2ydm(z)\in\text{CM}_{s}(\mathbb{H}) \Longleftrightarrow |(\nabla P_f)\circ\alpha(w)|^2\text{Im}\alpha(w) |\alpha'(w)|dm(w)\in\text{CM}_{0}(\mathbb{D}). \]
	Now, the result follows from \begin{align*}
		|(\nabla P_f)\circ\alpha(w)|^2\text{Im}\alpha(w) |\alpha'(w)|&=|(\nabla P_f)\circ\alpha(w)\cdot\alpha'(w)|^2 (1-|w|^2)\\
		&=|\nabla( P_f\circ\alpha)(w)|^2 (1-|w|^2)\\
		&=|\nabla P_{f\circ\alpha}(w)|^2 (1-|w|^2).
	\end{align*}
	
\end{proof}

Recall that a holomorphic function $f$ on $\mathbb{H}$  belongs to the space CMOA($\mathbb{H}$) if there exists some function $u\in\text{CMO}(\mathbb{R})$ whose Poission extension $P_{u}$ is precisely $f$. Then we have   the following result.
\begin{cor} \label{3.9}
	If $ f\in\operatorname{BMOA}(\mathbb{H})$, then the following statements are all equivalent:
	
 (1) $ f\in\operatorname{CMOA}(\mathbb{H}) $;
	
(2) $ f|_{\mathbb{R}}\in\operatorname{CMO}(\mathbb{R}) $;
	
(3) $ |f'(z)|^2 ydm(z)\in\operatorname{CM}_s(\mathbb{H}) $;
	
(4) $ f\circ\gamma^{-1}\in\operatorname{VMOA}(\mathbb{D}) $.
\end{cor}

We prepare some basic results on strongly vanishing Carleson measures for the proof of Theorem \ref{1.1} in the next section. We can obtain the following propositions by repeating the discussion to prove Lemmas 7.1, 7.3 and also Proposition 7.4  in our paper \cite{Sh}, where the cases $\operatorname{CM}(\mathbb{H})$ and $\operatorname{CM}_{0}(\mathbb{H})$ are considered. They also hold on the lower half plane $\mathbb{L}$.

\begin{prop}\label{3.10}

Let $\phi$ be analytic  on the upper half plane $\mathbb{H}$, $n\in\Bbb
N$, $\alpha>0$. Set $\lambda(z)=|\phi(z)|^n|y|^\alpha$ for $z=x+iy\in\mathbb{H}$. If  $\lambda\in \operatorname{CM}_{s}(\mathbb{H})$, then
for all compact sets $K$
\[\inf_{K\subset\mathbb{H}}\sup_{z\in\mathbb{H}/K}|\phi(z)|^n|y|^{\alpha+1}=0.\]
\end{prop}

\begin{prop}\label{3.11}
Let $\phi$ be a holomorphic function on the upper half plane $\mathbb{H}$, and for $\alpha>0$ set $\lambda_1(z)=|\phi(z)|^2|y|^{\alpha}$ and $\lambda_2(z)=|\phi'(z)|^2|y|^{\alpha+2}$. Then
  $\lambda_2\in \operatorname{CM}_{s}(\mathbb{H})$ if  $\lambda_1\in \operatorname{CM}_{s}(\mathbb{H})$.
\end{prop}

Let $A(\mathbb{H})$ denote the Banach space of functions $f$ holomorphic in $\mathbb{H}$ with norm
\[\|f\|_{A}\doteq \sup_{z\in\mathbb{H}} y^2|f(z)|<\infty,\] and $A_{0}(\mathbb{H})$ the closed subspace of  $A(\mathbb{H})$ which consists of those functions $f$ such that for compact sets $K$
\[\inf_{K\subset\mathbb{H}}\sup_{z\in\mathbb{H}/K}y^2|f(z)|=0.\]
Corresponding to this Banach space is the familiar Bloch space.
Recall that a  holomorphic function $f$ on $\mathbb{H}$  belongs to the Bloch space $B(\mathbb{H})$ if  \[\|f\|_{B}\doteq \sup_{z\in\mathbb{H}} y|f'(z)|<\infty,\]  and  the little Bloch space $B_{0}(\mathbb{H})$ is the closed subspace of $B(\mathbb{H})$  which consists of those functions $f$
such that  for compact sets $K$  \[\inf_{K\subset\mathbb{H}}\sup_{z\in\mathbb{H}/K}y|f'(z)|=0.\]
It is well known that  $f''\in A(\mathbb{H})$ if $f\in B(\mathbb{H})$, and $f''\in A_0(\mathbb{H})$ if $f\in B_0(\mathbb{H})$ (see \cite[Lemma 3.3]{STW}). We may define these spaces on $\mathbb{L}$, $\mathbb{D}$ and $\mathbb{D^*}$ in the same way.

\begin{cor} \label{3.12}
	Suppose $f\in \operatorname{CMOA}(\mathbb{H})$. Then  $f\in B_{0}(\mathbb{H})$.
\end{cor}
\begin{proof}
	By Corollary \ref{3.9}, $ |f'(z)|^2 ydm(z)\in\operatorname{CM}_s(\mathbb{H}) $. It follows from Proposition  \ref{3.10} that  $f\in B_{0}(\mathbb{H})$.
\end{proof}

\section{Boundary correspondence theorem}

We begin by stating a fundamental lemma, which  will be used  in the proof of Theorem \ref{1.1}.

We say that a homeomorphism   $h\in\text{Hom}^{+}(\mathbb{R})$  belongs to the class 	$\widetilde{\operatorname{SS}_{0}}(\mathbb{R})$ if it can be extened as a q.c. map of $\mathbb{H}$ onto itself with a complex dilatation $\mu$ such that $|\mu(z)|^2/ydm(z)\in\operatorname{CM}_{s}(\mathbb{H})$.

\begin{lem} \label{4.1}
	$\widetilde{\operatorname{SS}_{0}}(\mathbb{R})$ forms a group.
\end{lem}
\begin{proof}
	For any $h_{1}$,  $h_{2}\in\widetilde{\operatorname{SS}_{0}}(\mathbb{R})$, we have
	$\gamma\circ h_{1}\circ\gamma^{-1}$, $\gamma\circ h_{2}\circ\gamma^{-1}\in{\operatorname{SS}}(S^1)$ by Theorem \ref{3.7}. Since
	${\operatorname{SS}}(S^1)$ is a group, we have $\gamma\circ h_{1}\circ h^{-1}_{2}\circ\gamma^{-1}\in{\operatorname{SS}}(S^1)$.   By Theorem \ref{3.7} again, we obtain that $h_{1}\circ h^{-1}_{2}\in\widetilde{\operatorname{SS}_{0}}(\mathbb{R})$.
\end{proof}

 Theorem \ref{1.1} can be restated as follows.
\begin{thm}\label{4.2}
	$\operatorname{SS}_{0}(\mathbb{R})=\widetilde{\operatorname{SS}_{0}}(\mathbb{R})$.
\end{thm}

\begin{proof}
To prove $\operatorname{SS}_{0}(\mathbb{R})\subseteq\widetilde{\operatorname{SS}_{0}}(\mathbb{R}) $, we  will follow the setting developed by Semmes in \cite{S2} to construct a slight variant of the Beurling-Ahlfors extension and make some  adjustments so that it can be well adapted  to $\operatorname{SS}_{0}(\mathbb{R})$.

Suppose  $h$ belongs to $\operatorname{SS}_{0}(\mathbb{R})$. Define $ h_t$ by  \[h_t(x)=h(0)+\int_{0}^{x} h'(u)^tdu,\quad 0\leq t\leq1.\] Then $ t\log h'\in\text{CMO}(\mathbb{R}), $ so  $ (h')^t\in A_2$ and $ h_t\in \text{SS}_0(\mathbb{R})\subseteq  \text{SQS}(\mathbb{R}) $. For $s<t$,	
\[(h_t\circ h^{-1}_s)'(x)=(h'\circ h^{-1}_s(x))^t\cdot (h^{-1}_s)'(x)=(h'\circ h^{-1}_s(x))^{t-s}.\] Thus \[\|\log(h_t\circ h^{-1}_s)'\|_*=(t-s)\|\log h'\circ h^{-1}_s\|_*\leq C(t-s)\|\log h'\|_*,\] where  we have used that $ P_{h_s}$ is a bounded operator with norm independent of $s$. Clearly, we make $ \|\log(h_t\circ h^{-1}_s)'\|_* $ as small as we need by taking $t-s$ small enough. Fix a sufficiently large $n$,  set $ t_N=\frac{N}{n}$ and $ k_N=h_{t_N}\circ h^{-1}_{t_{N-1}}$ ($ 1\leq N\leq n $). Then $ k_N\in\text{SS}_0(\mathbb{R})$  and $\|\log k'_N\|_* $ is small,  and we can write $ h=k_n\circ k_{n-1}\circ\cdots\circ k_2\circ k_1$.

Let $\varphi$ and $\psi$ be smooth even function and odd function with support on $[-1,1]$, respectively, and $ \int \varphi(x)dx=-\int x\psi(x)dx=1 $. Set $ f_y(x)=y^{-1}f(y/x) $, and define $\rho_{N}: \mathbb{C}\rightarrow\mathbb{C}$ by
\begin{align*}
	\rho_{N}(x,y)&\doteq\varphi_y*k_{N}(x)+i\psi_y*k_{N}(x),\\
	\rho_{N}(x,0)&\doteq k_{N}(x).
\end{align*}
Let $\alpha(x)=\frac{1}{2}(1-ix)(\varphi(x)+i\psi(x))$, $ \beta(x) =\frac{1}{2}(1+ix)(\varphi(x)+i\psi(x))$ and $ \log k'_N=a $.   A brutal computation yields that
$ \int_{\mathbb{R}}\alpha(t)dt=0 $,  $\int_{\mathbb{R}}\beta(t)dt=1$   and  \[\overline{\partial}\rho_N(z)=\alpha_y*e^a(x),\quad \partial\rho_N(z)=\beta_y*e^a(x).\]
Note that $\|a\|_{*}$ is small enough, then (see \cite[p.250]{S2}) \[|\beta_y*e^a|\simeq|\exp(\beta_y*a)|, \quad \frac{1}{2y}\int_{x-y}^{x+y}|e^{a(u)-\beta_y*a(x)}-1|du\lesssim \|a\|_*. \]
Thus, \begin{align*}
	|\mu_N(z)|&\doteq\frac{|\alpha_y*e^a(x)|}{|\beta_y*e^a(x)|}\simeq |(\alpha_y*e^{a-\beta_y*a(x)})(x)|\\
	&=|\alpha_y*(e^{a-\beta_y*a(x)}-1)(x)|\\
	&\leq C\|a\|_*<1.
\end{align*}
Next, let's show that  $ |\mu_N(z)|^2/y $ belongs to $\operatorname{CM}_{s}(\mathbb{H})$.  We may assume that $ a_{3I}=0 $, otherwise replace $a $ with $a- a_{3I}$, which does not change $\mu_{N}$.
Since $ |\beta_y*e^a|^{-1}\simeq|\beta_y*e^{-a}| $,
\begin{align*}
	\int_{Q_{I}}\frac{|\mu_N(z)|^2}{y}dm(z)&\simeq \int_{I}\int_{0}^{|I|}|\alpha_y*e^a(x)|^2|\beta_y*e^{-a}(x)|^2y^{-1}dydx\\
	&\leq \Big(\int_{I}\Big(\int_{0}^{|I|}\frac{|\alpha_y*e^{a}(x)|^2}{y}dy\Big)^2dx\Big)^{\frac{1}{2}}\Big(\int_{I} \sup_{0<y<|I|}|\beta_y*e^{-a}(x)|^4dx\Big)^{\frac{1}{2}}\\
	&\doteq V_1 \cdot V_2.
\end{align*} Note that supp $\alpha$ and supp $\beta $ lie in $ \left[-1,1\right] $, and $ \int_{\mathbb{R}}\alpha(t)dt=0 $,    then for $(x,y)\in I\times(0, |I|) $ we have $ \left[x-y,x+y\right]\subseteq 3I $ and
\begin{align*}
	\alpha_y*e^{a}(x)&=\frac{1}{y}\int_{x-y}^{x+y}\alpha\Big(\frac{x-t}{y}\Big)(e^{a(t)}-1)\chi_{3I}(t)dt=\alpha_y*(e^{a}-1)\chi_{3I}(x) \\
	\beta_y*e^{-a}(x)&=\frac{1}{y}\int_{x-y}^{x+y}\alpha\Big(\frac{x-t}{y}\Big)e^{-a(t)}\chi_{3I}(t)dt=\beta_y*e^{-a}\chi_{3I}(x).
\end{align*}
By  $ L^4 $-boundedness of the Littlewood-Paley $ g $ function (see \cite{St}) and the Fact,
\begin{align*}
	V_1&=\Big(\int_{I}\Big(\int_{0}^{|I|}\frac{|\alpha_y*(e^{a}-1)\chi_{3I}(x)|^2}{y}dy\Big)^2dx\Big)^{\frac{1}{2}}\\
	&\lesssim\Big(\int_{3I} |e^{a(x)}-1|^4dx\Big)^{\frac{1}{2}}\leq \Big(\int_{3I} e^{4|a(x)|}|a|^4dx\Big)^{\frac{1}{2}},\\
	&\leq \Big(\int_{3I}|a(x)|^8dx\int_{3I} e^{8|a(x)|}dx\Big)^{\frac{1}{4}}\\
	&\lesssim \Big(\frac{1}{|I|}\int_{3I}|a(x)|^8dx\Big)^{\frac{1}{4}}|I|^{\frac{1}{2}}.
\end{align*}
By  $ L^4 $-boundedness of the maximal function (see  \cite{St}), \[V_2=\Big(\int_{I} \sup_{0<y<|I|}|\beta_y*e^{-a}\chi_{3I}(x)|^4dx\Big)^{\frac{1}{2}}\lesssim\Big(\int_{3I} e^{-4a(x)}dx\Big)^{\frac{1}{2}}\lesssim|I|^{\frac{1}{2}}.\]
Then \[ \frac{1}{|I|}\int_{I}\int_{0}^{|I|}\frac{|\mu_N(z)|^2}{y}dxdy \lesssim  \frac{V_1\cdot V_2}{|I|}\lesssim M^{1/4}_{8,a}(3I) \  (\lesssim\|a\|^{2}_{*}). \]
Hence,  $ |\mu_N|^2/y\in\text{CM}_s(\mathbb{H}) $ by Lemma \ref{3.1}. In addition, $\rho_{N}$ is a q.c. map
by Semmes, then $k_{N}\in\widetilde{\operatorname{SS}_{0}}(\mathbb{R})$ for $1\leq N\leq n$.  By Lemma \ref{4.1}, we get $h=k_n\circ k_{n-1}\circ\cdots\circ k_2\circ k_1\in\widetilde{\operatorname{SS}_{0}}(\mathbb{R})$.  The proof of  $\operatorname{SS}_{0}(\mathbb{R})\subseteq\widetilde{\operatorname{SS}_{0}}(\mathbb{R}) $ has been completed.

Now, we need to prove  $\widetilde{\operatorname{SS}_{0}}(\mathbb{R})\subseteq \operatorname{SS}_{0}(\mathbb{R})$, which is a tricky part of the proof.
Let's first review some of the notations again. As before, $\gamma(z)=(z-i)/(z+i)$ and $\gamma_{\zeta}(z)=(z-\zeta)/(z-\bar{\zeta})$, $\zeta=\xi+i\eta\in\mathbb{H}$.  Denote by $S_{f}$ the Schwarzian derivative of a locally univalent function $f$, which is defined as
\begin{equation}
	S_{f}\doteq N'_{f}-\frac{1}{2}N^{2}_{f}, \quad N_{f}\doteq(\log f')'.
\end{equation}

If $h\in\widetilde{\operatorname{SS}_{0}}(\mathbb{R})$, then  $h$ admits a  q.c. extension $H$   to $\mathbb{H}$  with  complex dilatation $\mu$ such that $|\mu(z)|^2/ydm(z)\in\operatorname{CM}_{s}(\mathbb{H})$. Also, there exists a unique  global q.c. map $\rho$   (in the  normalized sense)  which has complex   dilatation $\mu$ on  $\mathbb{H}$  and is conformal on $\mathbb{L}$.
Setting $f_{\mu}=\rho|_{\mathbb{H}}\circ H^{-1}$ and $g_{\mu}=\rho|_{\mathbb{L}}$, then $h=(f_{\mu}^{-1}\circ g_{\mu})|_{\mathbb{R}}$ clearly.
We call the pair $\{f_{\mu}, g_{\mu}\}$ the normalized solution  of the conformal welding problem to $h$.
Set $\tilde{\rho}=\rho\circ\gamma^{-1}$.
It is easy to see that the complex dilatation $\mu_{\tilde{\rho}}$ of $\tilde{\rho}$ satisfies  $|\mu_{\tilde{\rho}}(w)|^2/(1-|w|^2)\in\text{CM}_{0}(\mathbb{D})$ by Theorem \ref{3.7}, and  $|S_{g_{\mu}\circ\gamma^{-1}}(w)|^2(|w|^2-1)^3\in\operatorname{CM}_{0}(\mathbb{D}^*)$ by  results on $T_{v}(S^1)$ (see \cite[Theorem 4.1]{SW}). Then, we get $|S_{g_{\mu}}(z)|^2|y|^3\in\operatorname{CM}_{s}(\mathbb{L})$ by Theorem \ref{3.7} again, and $S_{g_{\mu}}\in A_{0}(\mathbb{L})$ by Proposition \ref{3.10}.
Therefore, we get $\log g_{\mu}'\in B_{0}(\mathbb{L})$ in terms of the known results of the so-called little Teichm\"{u}ller space (see \cite[Theorem 3.2]{STW}).
Set $g(z)=\overline{g_{\mu}(\bar{z})}$ for $z\in\mathbb{H}$, then 	$|S_{g}(z)|^2y^3dm(z)\in\operatorname{CM}_{s}(\mathbb{H})$, $S_{g}\in A_{0}(\mathbb{H})$
and $\log g'\in B_{0}(\mathbb{H})$.

Let's first  prove that $\log g'\in \operatorname{CMOA}(\mathbb{H})$, or equivalently, $|N_{g}(z)|^{2}ydm(z)\in\operatorname{CM}_{s}(\mathbb{H})$ by Corollary \ref{3.9}. Then we by Proposition \ref{3.6} must show that for each $\varepsilon>0$ there exists a compact set $K_{n_{0}}\subset\mathbb{H}$ such that
\begin{equation}\label{4.0}
	\sup_{\zeta\in\mathbb{H}/K_{n_{0}}} \int_{\mathbb{H}}|N_{g}(z)|^{2}y|\gamma'_{\zeta}(z)|dm(z)<\varepsilon.
\end{equation}

	 Set  $g_{\zeta}=g\circ\gamma^{-1}_{\zeta}$ so that $g=g_{\zeta}\circ\gamma_{\zeta}$. A  direct computation shows that
	\begin{equation}\label{4.1.}
		N_{g}=N_{g_{\zeta}\circ\gamma_{\zeta}}=(N_{g_{\zeta}}\circ\gamma_{\zeta})\gamma'_{\zeta}+N_{\gamma_{\zeta}}=(N_{g_{\zeta}}\circ\gamma_{\zeta}-N_{\gamma^{-1}_{\zeta}}\circ\gamma_{\zeta})\gamma'_{\zeta}.
	\end{equation} On the other hand, for  each  holomorphic function $\phi$ in $\mathbb{D}$,  a short computation gives (see also \cite[Theorem 4.28]{Zh}).
	\begin{equation} \label{4.2.}
	\int_{\mathbb{D}}|\phi(w)|^2(1-|w|^2)dm(w)\simeq|\phi(0)|^2+\int_{\mathbb{D}}|\phi'(w)|^2(1-|w|^2)^{3}dm(w).
	\end{equation}
	By \eqref{4.1.} and \eqref{4.2.}, we write
	\begin{equation}
		\begin{aligned}
			J(\zeta)&\doteq \int_{\mathbb{H}}|N_{g}(z)|^{2}y|\gamma'_{\zeta}(z)|dm(z)\\
			&=\frac{1}{2} \int_{\mathbb{H}}(1-|\gamma_{\zeta}|^2) \big|(N_{g_{\zeta}}\circ\gamma_{\zeta}-N_{\gamma^{-1}_{\zeta}}\circ\gamma_{\zeta})\gamma'_{\zeta}\big|^2dm(z)\\
			&=\frac{1}{2}  \int_{\mathbb{D}} (1-|w|^2)\big| N_{g_{\zeta}}(w) -N_{\gamma^{-1}_{\zeta}}(w)\big|^2dm(w)\\
			&\simeq\big| N_{g_{\zeta}}(0) -N_{\gamma^{-1}_{\zeta}}(0)\big|^2+\int_{\mathbb{D}} (1-|w|^2)^{3}\big| N'_{g_{\zeta}}(w) -N'_{\gamma^{-1}_{\zeta}}(w)\big|^2dm(w).
		\end{aligned}
	\end{equation}
We can now express $J$ as the sum of the following integrals by \eqref{4.1.} again.
\begin{equation}
	\begin{aligned}
	J(\zeta)\simeq\big| \eta N_{g}(\zeta)\big|^2 &+\int_{\mathbb{D}} (1-|w|^2)^{3}\big|S_{g_{\zeta}}(w)+\frac{1}{2} \big(N_{g_{\zeta}}(w) -N_{\gamma^{-1}_{\zeta}}(w)\big)\big|^2dm(w).\\
	=\big| \eta N_{g}(\zeta)\big|^2 &+ \int_{\mathbb{H}} |\gamma'_{\zeta}(z)|y^3 \big|S_{g}(z)+\frac{1}{2}N_{g}(z)\big(N_{g}(z)-2N_{\gamma_{\zeta}}(z)\big)\big|^{2}dm(z)\\
	\lesssim \big| \eta N_{g}(\zeta)\big|^2 &+ \int_{\mathbb{H}} |\gamma'_{\zeta}(z)|y^3 \big|S_{g}(z) \big|^{2}dm(z)\\
	&+\int_{\mathbb{H}}|\gamma'_{\zeta}(z)|y^3 \big(|N_{g}(z)|^{2}+|N_{\gamma_{\zeta}}(z)|^{2}\big)|N_{g}(z)|^{2}dm(z)\\
	\doteq J_{1}(\zeta)+J&_{2}(\zeta)+J_{3}(\zeta).
	\end{aligned}
\end{equation}
Since $|S_{g}(z)|^2y^3dm(z)\in\operatorname{CM}_{s}(\mathbb{H})$
and $\log g'\in B_{0}(\mathbb{H})$, there exists an  $n_{1}>\max\{\frac{1}{\varepsilon}, 10^{3}\}$ and the corresponding compact set $K_{n_{1}}=[-n_{1},n_{1}]\times[{n_{1}}^{-1},n_{1}]$  such that
 \begin{equation}
\sup_{\zeta\in\mathbb{H}/K_{n_{1}}}J_{1}(\zeta)<\varepsilon,\,\sup_{\zeta\in\mathbb{H}/K_{n_{1}}}J_{2}(\zeta)<\varepsilon.
\end{equation} Also, we for $J_{3}(\zeta)$ have
\begin{equation}
	\begin{aligned}
	J_{31}(\zeta)\doteq&\int_{\mathbb{H}/K_{n_{1}}}|\gamma'_{\zeta}(z)|y^3 \big(|N_{g}(z)|^{2}+|N_{\gamma_{\zeta}}(z)|^{2}\big)|N_{g}(z)|^{2}dm(z)\\
		&\leq \varepsilon^2\int_{\mathbb{H}/K_{n_{1}}}y|\gamma'_{\zeta}(z)| \big(|N_{g}(z)|^{2}+|N_{\gamma_{\zeta}}(z)|^{2}\big)dm(z)\\
		&<\varepsilon^2\int_{\mathbb{H}}y|\gamma'_{\zeta}(z)| \big(|N_{g}(z)|^{2}+|N_{\gamma_{\zeta}}(z)|^{2}\big)dm(z)\\
		&\lesssim\varepsilon^2\big(\|\log g'\|_{*}+\|\log \gamma'_{\zeta}\|_{*}\big).
	\end{aligned}
\end{equation}
Note that $\log|x|$ is a canonical BMO function, we  by translation invariance of BMO have $\log \gamma'_{\zeta}(x)= \log2\eta-2\log(x-\bar{\zeta})\in\text{BMO}$ ($\zeta$ has been killed by the BMO norm). Consquently,  $J_{31}(\zeta)\lesssim\varepsilon^2$.

For another part of $J_{3}(\zeta)$, write
\begin{equation}
\begin{aligned}
	J_{32}(\zeta)&\doteq\int_{K_{n_{1}}}|\gamma'_{\zeta}(z)|y^3 \big(|N_{g}(z)|^{2}+|N_{\gamma_{\zeta}}(z)|^{2}\big)|N_{g}(z)|^{2}dm(z)\\
	&=\int_{K_{n_{1}}}{|\gamma'_{\zeta}(z)|}y^{-1}\big(|yN_{g}(z)|^{4}+|yN_{\gamma_{\zeta}}(z)|^{2}|yN_{g}(z)|^{2}\big)dm(z)\\
	&\lesssim\int_{K_{n_{1}}}\frac{\eta dm(z)}{y\big((x-\xi)^2+(y+\eta)^2\big)}\\
	&=\int_{n^{-1}_{1}}^{n_{1}}\frac{\eta}{y}dy\int_{-n_{1}}^{n_{1}}\frac{dx}{(x-\xi)^2+(y+\eta)^2}\\
	&=\int_{n^{-1}_{1}}^{n_{1}}\frac{\eta}{y(y+\eta)}\Big(\arctan\frac{n_{1}+\xi}{y+\eta}-
	\arctan\frac{\xi-n_{1}}{y+\eta}\Big)dy.
\end{aligned}
\end{equation}

Select  a priori compact set $K_{n_{0}}=[-(n^{4}_{1}+n_{1}),n^{4}_{1}+n_{1}]\times [n^{-2}_{1},n^{2}_{1}\log n_{1}]$.
It remains to check that for any $\zeta\in\mathbb{H}/K_{n_{0}}$,
\begin{equation}
J_{32}(\zeta)\doteq  \int_{n^{-1}_{1}}^{n_{1}}\frac{\eta}{y(y+\eta)}\Big(\arctan\frac{n_{1}+\xi}{y+\eta}-
	\arctan\frac{\xi-n_{1}}{y+\eta}\Big)dy<\varepsilon.
\end{equation}
To this end, we divide $\mathbb{H}/K_{n_{0}}$ into three categories so that $\mathbb{H}/K_{n_{0}}=Q_{1}\cup Q_{2}\cup Q_{3} $.

$\bullet$ $Q_{1}=\{(\xi,\eta):-\infty<\xi<\infty, \  0<\eta<n^{-2}_{1}\}$.

$\bullet$ $Q_{2}=\{(\xi,\eta):-\infty<\xi<\infty, \ \eta>n^{2}_{1}\log n_{1}\}$.

$\bullet$ $Q_{3}=\{(\xi,\eta):|\xi|>n^{4}_{1}+n_{1}, \ n^{-2}_{1}<\eta< n^{2}_{1}\log n_{1}\}$.

For  $\zeta\in Q_{1}$,
 \begin{equation}
\begin{aligned}
J_{32}(\zeta)&\leq\int_{n^{-1}_{1}}^{\infty}\frac{\pi\eta}{y(y+\eta)}dy\\
 &=\pi\log(1+\eta n_{1})\\
 &\leq \pi n^{-1}_{1}<\pi\varepsilon.
\end{aligned}	
\end{equation}

For  $\zeta\in Q_{2}$,
	\begin{equation}
		\begin{aligned}
			J_{32}(\zeta)&\leq\int_{n^{-1}_{1}}^{n_{1}}\frac{\eta}{y(y+\eta)}\cdot \frac{2n_{1}}{y+\eta}dy\\
			&\leq\int_{n^{-1}_{1}}^{n_{1}}\frac{1}{y}\cdot \frac{2n_{1}}{n^{2}_{1}\log n_{1}}dy\\
			&=\frac{4\log n_{1}}{n_{1}\log n_{1}}<4\varepsilon.
		\end{aligned}	
	\end{equation}

For  $\zeta\in Q_{3}$,	we assume that $\xi>n^{4}_{1}+n_{1}$ without losing generality. Note that
$n^{-2}_{1}<\eta< n^{2}_{1}\log n_{1}$, we have
		\begin{equation}
		\begin{aligned}
			J_{32}(\zeta)&\leq\int_{n^{-1}_{1}}^{n_{1}}\frac{\eta}{y(y+\eta)}\cdot \frac{2n_{1}}{y+\eta}\cdotp \frac{1}{1+\big(\frac{\xi-n_{1}}{y+\eta}\big)^2}dy\\
			&<2\int_{n^{-1}_{1}}^{n_{1}}\frac{n_{1}}{y^{2}}\cdot \frac{(n_{1}+n^{2}_{1}\log n_{1})^{2}}{n^{8}_{1}}dy\\
			&\lesssim \frac{1}{n^{2}_{1}}\int_{n^{-1}_{1}}^{n_{1}}\frac{1}{y^{2}}dy\lesssim\varepsilon.
		\end{aligned}	
	\end{equation}
	This concludes the proof of \eqref{4.0}, that is,  $\log g'\in \operatorname{CMOA}(\mathbb{H})$, or equivalently, $\log g_{\mu}'\in \operatorname{CMOA}(\mathbb{L})$.
	
	Now by  Lemma \ref{4.1},  $h^{-1}\in\widetilde{\operatorname{SS}_{0}}(\mathbb{R})$. Clearly, $h^{-1}=(g_{\mu}^{-1}\circ f_{\mu})|_{\mathbb{R}}=(g^{-1}\circ f)|_{\mathbb{R}}$ with $f(z)=\overline{f_{\mu}(\bar{z})}$ for $z\in\mathbb{L}$. Thus the pair $\{g, f\}$ is the normalized solution  of the conformal welding problem to $h^{-1}$. Consequently, by what we have proved we obtain $\log f'\in \operatorname{CMOA}(\mathbb{L})$, or equivalently, $\log f_{\mu}'\in \operatorname{CMOA}(\mathbb{H})$.
	
Now on the real line,   $g_{\mu}=f_{\mu}\circ h$, which implies that  $g'_{\mu}=(f'_{\mu}\circ h)h'$ and so	
\begin{equation} 
	\log h'=\log g_{\mu}'-\log f'_{\mu}\circ h.
\end{equation}
By Corollary \ref{3.9},
both $\log f_{\mu}'|_{\mathbb{R}}$ and $\log g_{\mu}'|_{\mathbb{R}}$ lie in 	$\text{CMO}(\mathbb{R})$. Since  the pull-back operator $P_{h}$ preserves $\text{CMO}(\mathbb{R})$, it follows  that $\log h'\in\text{CMO}(\mathbb{R})$.
This completes the proof of  $\widetilde{\operatorname{SS}_{0}}(\mathbb{R})\subseteq \operatorname{SS}_{0}(\mathbb{R})$  and of Theorem \ref{4.2} as well.

\end{proof}

In the proof, we dispose of  logarithmic derivative of the conformal mapping without conformal invariance. Actually we have  the following result parallel to the VMO-Teichm\"{u}ller theory on the unit circle (see\cite[Theorem 4.1]{SW}).

\begin{cor}\label{4.3}
	 Let $g$ be a conformal mapping on the lower half plane $\mathbb{L}$   and $h=f^{-1}\circ g$ be the corresponding
	quasisymmetric conformal welding. Then the following statements are equivalent:
	
	(1) $h\in\operatorname{SS}_{0}(\mathbb{R})$;
	
	(2) $g$ admits a quasiconformal extension to the upper half plane $\mathbb{H}$ with  complex dilatation $\mu$ such that $|\mu(z)|^2/ydm(z)\in\operatorname{CM}_{s}(\mathbb{H})$;
	
	(3) $|S_{g}(z)|^{2}|y|^{3}dm(z)\in\operatorname{CM}_{s}(\mathbb{L})$;
	
	(4) $\log g'\in\operatorname{CMOA}(\mathbb{L})$.
\end{cor}

\begin{proof}
	$(1)\Longleftrightarrow(2)$ is a restatement of  Theorem \ref{1.1}, while $(3)\Longrightarrow(4)$ is contained in the proof of   Theorem \ref{4.2}.
	$(2)\Longleftrightarrow(3)$ is due to the conformal invariance of $\operatorname{CM}_{s}(\mathbb{H})$, as have been used  in in the proof of   Theorem \ref{4.2}.
	
	$(4)\Longrightarrow(3)$ follows from a standard argument.  Noting that  $\log g'\in B(\mathbb{L})$ by Proposition \ref {3.10}, we obtain
	\begin{equation}
		|S_{g}(z)|^{2}|y|^{3}\lesssim |N'_{g}(z)|^{2}|y|^{3}+|N_{g}(z)|^{2}|y|.
	\end{equation}
Since $\log g'\in\operatorname{CMOA}(\mathbb{L})$, $|N_{g}(z)|^{2}|y|\in\operatorname{CM}_{s}(\mathbb{L})$ by Corollary \ref{3.9}, which implies $|N'_{g}(z)|^{2}|y|^{3}\in\operatorname{CM}_{s}(\mathbb{L})$ by Proposition \ref{3.11}. Thus, $|S_{g}(z)|^{2}|y|^{3}dm(z)\in\operatorname{CM}_{s}(\mathbb{L})$.
\end{proof}

Finally, we point out that the paper is an expanded version of Chapter 6 of the Thesis (Ph.D.) \cite{Liu} by the first named author. After the major part of this research  was  completed, the authors got to know that in his Thesis \cite{S0} Semmes has already proved the part $(1)\Longleftrightarrow(4)$ of Corollary \ref {4.3} by a very different approach.


\begin{thebibliography}{99}\footnotesize\addtolength{\itemsep}{-2ex}
\bibitem{AZ} Astala, K.,  Zinsmeister, M., Teichm\"uller spaces and BMOA, \emph{Math. Ann.}, 1991, \textbf{289}, 613--625.
\bibitem{BA} Beurling, A., Ahlfors, L.V., The boundary correspondence under quasiconformal mappings, \emph{Acta Math.}, 1956, {\bf 96},  125--142.
\bibitem{BLS} Bourdaud, G., Lanza de Cristoforis, M., Sickel, W., Functional calculus on BMO and related spaces, \emph{J. Funct. Anal.},  2002, \textbf{189}(2), 515--538.
\bibitem{Bo} Bourdaud G., Remarques sur certains sous-espaces de $ BMO ({\mathbb {R}}^ n) $ et de $ bmo ({\mathbb {R}}^ n) $, \emph{Ann. Inst. Fourier}, 2002, \textbf{52}(4), 1187--1218.
\bibitem{Ca} Carleson, L., On mappings, conformal at the boundary, J. Analyse Math. 1967, \textbf{52}(4),  1--13.
\bibitem{CF} Coifman, R., Fefferman, C., Weighted norm inequalities for maximal functions and singular integrals, \emph{Studia Math.},  {\bf 51}, 241--250, 1974.
\bibitem{CW}  Coifman, R., Weiss, G., Extensions of Hardy spaces and their use in analysis, \emph{Bull. Amer. Math. Soc.}, 1977, \textbf{83}, 569--645.
\bibitem{FKP} Fefferman, R., Kenig, C., Pipher. J., The theory of weights and the Dirichlet problems for elliptic equations, \emph{Ann. of Math.} (2), 1991, \textbf{134}, 65--124.
\bibitem{GS}  Gardiner, F.P., Sullivan, D., Symmetric structures on a closed curve, \emph{Amer. J. Math.}, \textbf{114}, 1992, 683--736.
\bibitem{Gar} Garnett, J.B., Bounded Analytic Functions, {Academic Press}, New York, 1981.
\bibitem{Gi} Girela, D., Analytic functions of bounded mean oscillation, \emph{Complex function spaces}, 2001, \textbf{4}, 61--170.
\bibitem{Jon} Jones, P.W., Homeomorphisms of the line which preserve BMO, \emph{Ark. Mat.}, 1983, \textbf{21}, 229--231.
\bibitem{Liu} Liu, T., The Cauchy integral on chord-arc curves and the related research, Thesis (Ph.D.)-Soochow University, 2023.
\bibitem{Ne} Neri, U., Fractional  integration  on the space  $ H^1 $ and its dual, \emph{Studia Math.}, 1975, \textbf{53}, 175--189.
\bibitem{S0} Semmes, S., The Cauchy integral and related operators on smooth curves, Thesis (Ph.D.)-Washington University in St. Louis, 1983, 1--133.
\bibitem{S2} Semmes, S., Quasiconformal mappings and chord-arc curves, \emph{Tran. Amer. Math. Soc.}, {\bf 306},  233--263, 1988.
\bibitem{Sh} Shen, Y., VMO-Teichm\"uller space on the real line, \emph{Ann. Fenn. Math.}, 2022, \textbf{47}(1), 57--82.
\bibitem{STW} Shen, Y., Tang, S., Wu, L.,  Weil-Petersson and little Teichm\"uller spaces on the real line. \emph{Ann. Fenn. Math.}, 2018, \textbf{43}(2), 935--943.
\bibitem{SW}  Shen, Y., Wei, H., Universal Teichm\"uller space and BMO, \emph{Adv. Math.}, 2013, \textbf{234}, 129--148.
\bibitem{St} Stein, E. M., Singular Integrals and Differentiability Properties of functions, Princeton Univ. Press, Princeton, N.J., 1970.
\bibitem{Uc} Uchiyama, A.,  On the compactness of operators of Hankel type,  \emph{Tohoku Math. J.}, 1978, \textbf{30},  163--171.
\bibitem{WM1} Wei, H., Matsuzaki, K., Symmetric and strongly symmetric homeomorphisms on the real line with non-symmetric inversion, \emph{Anal. Math. Phys.}, 2021, \textbf{11}(2), 1--17.
\bibitem{WM2} Wei, H., Matsuzaki, K., Beurling-Ahlfors extension by heat kernel, $ A_\infty $ weights for VMO, and vanishing Carleson measures, \emph{Bull. Lond. Math. Soc.}, 2021, \textbf{53}(3), 723--739.
\bibitem{WM3}  Wei, H., Matsuzaki, K., Strongly symmetric homeomorphisms on the real line with uniform continuity, \emph{Indiana Univ. Math. J.}, 2023, \textbf{72}, 1553-1576.
\bibitem{Zh} Zhu, K., Operator Theory in Function Spaces, Amer. Math. Soc. Press, Providence,  2007.


\end{thebibliography}
\end{document}